\documentclass{daj}

\usepackage{amsmath,amssymb,amsthm}
\usepackage{graphicx}
\usepackage{caption}
\usepackage{subcaption}
\usepackage{wrapfig}

\def\R{\mathbb{R}}
\def\C{\mathbb{C}}
\def\N{\mathbb{N}}
\def\Z{\mathbb{Z}}

\def\area{\operatorname{\bf area}}

\def\Norm{\operatorname{\bf N}}
\def\argmin{\operatorname{argmin}}
\def\argmax{\operatorname{argmax}}
\def\dist{\operatorname{d}}

\def\Re{\operatorname{Re}}
\def\Im{\operatorname{Im}}
\def\SL{\operatorname{SL}}
\def\epsilon{\varepsilon}
\def\cC{\mathcal{C}}
\def\cE{\mathcal{E}}
\def\cH{\mathcal{H}}
\def\cL{\mathcal{L}}
\def\a{\operatorname{a}}
\def\sgn{\operatorname{sgn}}
\def\dist{\operatorname{\bf dist}}
\def\Hdim{\operatorname{Hdim}}
\newcommand\DZ[1]{\langle\!\langle#1\rangle\!\rangle}

\newtheorem{theorem}{Theorem}
\newtheorem{proposition}[theorem]{Proposition}
\newtheorem{lemma}[theorem]{Lemma}
\newtheorem{corollary}[theorem]{Corollary}
\newtheorem{definition}[theorem]{Definition}
\newtheorem{remark}[theorem]{Remark}

\dajAUTHORdetails{%
  title = {From a Packing Problem to Quantitative Recurrence in $[0,1]$ and the Lagrange Spectrum of Interval Exchanges},
  author = {M. Boshernitzan, and V. Delecroix},
  plaintextauthor = {Michael Boshernitzan, and Vincent Delecroix},
  plaintexttitle = {From a Packing Problem to Quantitative Recurrence in [0,1] and Lagrange spectrum of Interval Exchanges},
  runningtitle = {From a Packing Problem to Quantitative Recurrence},
  copyrightauthor = {M. Boshernitzan, and V. Delecroix},
  keywords = {dynamical systems, diophantine approximation, recurrence, interval exchange transformation},
}   

\dajEDITORdetails{%
   year={2017},
   number={10},
   received={2 December 2016},   
   published={2 June 2017},  
   doi={10.19086/da.1749},       
}   

\begin{document}

\begin{frontmatter}[classification=text]


\author[mb]{Michael Boshernitzan}
\author[vd]{Vincent Delecroix}

\begin{abstract}
In this work, we use use a solution to a packing problem in the plane to study
recurrence of maps on the interval [0,1]. First of all, we prove that $1/\sqrt{5}$
is the optiaml recurrence rate of measurable applications of the interval. Secondly,
we analyze the bottom of the Lagrange spectrum of interval exchange transformations.
\end{abstract}
\end{frontmatter}

\maketitle

\section{Introduction}
\subsection{Recurrence rate and Lagrange constants}
We study recurrence of maps of the unit interval that preserve Lebesgue measure.
Let $T\colon [0,1)\to [0,1)$ be a measurable map of the unit interval into itself.
The recurrence rate of $T$ at $x\in [0,1)$ is defined as
\[
r(T,x) = \liminf_{n \to \infty}\ n\, |T^n x-x|\in[0,\infty].
\]

Denote by $R_\alpha: [0,1) \to [0,1)$ the rotation by the angle $\alpha\in\R$,
that is 
$R_\alpha(x) = x+\alpha \pmod 1$.
Since the group of such rotations acts transitively on $[0,1)$  by local isometries (apart from at $x=0$), we note that \mbox{$r(R_\alpha,x)=r(R_\alpha, \frac12)$}, for all $\alpha$ and $x \not= 0$.
Set  $r(R_\alpha):=r(R_\alpha, \frac12)$.
One verifies that, for every $\alpha\in\R$,
\[
r(R_{\alpha})=\liminf_{q\to+\infty\atop q\in\Z}\ q\,\DZ{q\alpha},
\]
where  $\DZ{t}=\min_{n\in\Z}\limits|t-n|$ stands 
for the distance of $t\in\R$ to the nearest integer.

We will provide two generalizations of the following well known result.
\begin{theorem}[Hurwitz \cite{Hurwitz}] \label{thm:Hurwitz_rotation}
For any real number $\alpha$ we have $r(R_\alpha) \leq \frac{1}{\sqrt{5}}$.
Moreover, $r(R_\alpha) = \frac{1}{\sqrt{5}}$ if and only if the continued 
fraction expansion of $\alpha$  is eventually constant and equal to $1$.
(In particular, $r(R_\beta)=\frac{1}{\sqrt{5}}$ for 
$\beta = \frac{\sqrt{5}-1}{2} = [0; 1, 1, 1, ...]$).
On the other hand, if $r(R_\alpha) \not= \frac{1}{\sqrt{5}}$ then $r(R_\alpha) \leq \frac{1}{2 \sqrt{2}}$.
\end{theorem}
In other words, the inequality $r(R_{\alpha},x)\leq\frac1{\sqrt5}$  holds
 for all rotations $R_\alpha\colon [0,1) \to [0,1)$  and $x\neq0$.
Note that Hurwitz \cite{Hurwitz} also proved that the inequality $q\, \DZ{q\alpha}<\frac1{\sqrt5}$
has infinitely many solutions in integers $q\geq1$.
 
\emph{The Lagrange constant of} $\alpha\in\R$ is defined as follows:
\begin{equation} \label{eq:lagrange_rot_dim1}
L(\alpha) = \frac1{r(R_{\alpha})}=\limsup_{q \to \infty}\ \frac{1}{q\,\DZ{q\alpha}}
\in\big[\sqrt5,\infty\big], 
\end{equation}
and \emph{the (classical) Lagrange spectrum} is defined as the set of possible
finite Lagrange constants
\[
\cL_1 := \big\{L(\alpha):\, \alpha \in [0,1) \big\}\setminus\{+\infty\}.
\]
An equivalent way of stating Hurwitz's Theorem (Theorem~\ref{thm:Hurwitz_rotation}) is
\begin{equation}\label{eq:LagSp1}
\cL_1 \cap \big[0, 2\sqrt2\big) = \big\{\sqrt5\,\big\}.
\end{equation}

\subsection{Two generalizations of Hurwitz's Theorem}
The first proposed generalization is that the inequality $r(T,x) \leq
\frac{1}{\sqrt{5}}$ in Hurwitz's Theorem actually holds almost surely for all 
Lebesgue measure preserving transformations  $T\colon [0,1) \rightarrow [0,1)$
(and not just for the rotations $T=R_{\alpha}$).
\begin{theorem}\label{thm:rec01}
Let $T\colon [0,1) \rightarrow [0,1)$ be a measurable map of the unit interval
which preserves the Lebesgue measure~$\mu$.
Then, $r(T,x) \leq \frac1{\sqrt5}$, for $\mu$-almost every $x \in [0,1)$.
\end{theorem}

The above theorem provides the optimal constant in the quantitative recurrence
result in~\cite[Theorem 2.1]{Boshernitzan-recurrence} where, under the
conditions of Theorem \ref{thm:rec01}, a weaker inequality was
established (see Theorem \ref{thm:rec2} below).

Note that quantitative Poincar\'e recurrence results are possible in the more 
general settings of transformations of arbitrary metric spaces having finite 
Hausdorff dimension: see \cite{Boshernitzan-recurrence} 
and the discussion following Theorem~\ref{thm:rec2} below.

\bigskip

Our second generalization of Theorem~\ref{thm:Hurwitz_rotation} 
is related to interval exchange transformations.
An \emph{interval exchange transformation} (or \emph{i.e.t.} for short) is a bijection 
from an interval to itself that is a piecewise translation on finitely many intervals. 
More precisely,
given a permutation $\pi \in S_k$ and a vector $\lambda \in \R_+^k$,
we define a $k$-i.e.t. $T_{\pi,\lambda}$ as
\[
T_{\pi,\lambda}(x) = x - \sum_{j < i} \lambda_j + \sum_{\pi(j) < \pi(i)} \lambda_j,
\quad \text{ if \ $x \in [\lambda_1 + \ldots + \lambda_{i-1}, \lambda_1 + \ldots + \lambda_i)$}.
\]
Note that rotations are exactly the 2-i.e.t.s with permutations $\pi=(2,1)$.

The \emph{singularities} of $T$ are the points $x_i = \sum_{j \leq i}
\lambda_j$ for $i=1,\dots,k-1$. An i.e.t. $T_{\pi,\lambda}$ is said to be \emph{without connection} if there is
no pair of singularities $x$ and $y$ of $T$ such that $T^n x = y$ for some $n \geq 1$. It
was shown by Keane~\cite{Keane1975} that this condition implies the minimality
of the transformation $T$. 

Given an i.e.t. $T = T_{\pi,\lambda}$ that satisfies the Keane condition, its
$n$-th iterate $T^n$ is also an i.e.t. but on $(k-1)n+1$ intervals. Let
$\epsilon_n(T)$ be the smallest length of any of the intervals of $T^n$. In particular
$\epsilon_0(T) = 1$ and $\epsilon_1(T) = \min \lambda_i$. Note that the
number $\epsilon_n(T)$ can alternatively be defined as the minimum distance
between the $n$-th first preimages of the singularities together with $0$ and $1$.

For $T=T_{\pi,\lambda}$ we define
\[
\cE(T) := \liminf_{n \to \infty}\, \frac{n \epsilon_n(T)}{|\lambda|}
\quad \text{where } \ |\lambda| = \lambda_1 + \ldots + \lambda_k.
\]
The value $L(T) := \cE(T)^{-1}$ is called the \emph{Lagrange constant} of $T$.
It generalizes the Lagrange constant for the rotations $R_\alpha$. We also
recall that if $\pi \in S_k$ is \emph{irreducible} (or \emph{indecomposable})
then for Lebesgue-almost every $\lambda$ the Lagrange constant of
$T_{\pi,\lambda}$ is infinite.

The Lagrange spectrum of i.e.t.s was introduced by S.~Ferenczi
in~\cite{Ferenczi-lagrange} under the name \emph{lower Boshernitzan-Lagrange spectrum}. It
was then further studied  by P. Hubert, L. Marchese and C. Ulcigrai in~\cite{HubertMarcheseUlcigrai}.
The \emph{Lagrange spectrum} of the $k$-i.e.t.s is the following set of values:
\[
\cL_{k-1} = \{L(T)\colon\ \text{$T$ is a $k$-i.e.t. satisfying the Keane condition and $L(T) < \infty$}\}.
\]
Because $T^n$ is made of $(k-1)n+1$ intervals, $\inf\cL_{k-1} \geq k-1$.
In~\cite{HubertMarcheseUlcigrai}, the better bound 
$\inf\cL_k \geq \frac\pi2 k$ is established. 
We prove the following.
\begin{theorem} \label{thm:lagrange}
There exists a constant $\epsilon_0 > 0$ such that for any $d \geq 1$
\[
\cL_d \cap \left[0,\ d \sqrt{5} + \frac{\epsilon_0}{d}\right] = \left\{d\sqrt{5}\right\}.
\]
Moreover, for any permutation $\pi \in S_{d+1}$ such that $\pi(1) = d+1$, the length
vector $\lambda =  (\frac{\sqrt{5}+1}{2},1,1,\ldots,1)$ is such that
$T_{\pi,\lambda}$ satisfies the Keane condition and $L(T_{\pi,\lambda}) = d\sqrt{5}$.
\end{theorem}
The proof of Theorem \ref{thm:lagrange} will be given in
Section~\ref{sec:lagrange_iet}. We will actually completely characterize the
$(d+1)$-i.e.t.s $T$ such that $L(T) = d \sqrt{5}$. Note that the case $d=1$ is given
by Hurwitz's theorem (see in particular~\eqref{eq:LagSp1}) and that the case $d=2$ was proven
in~\cite[Theorem~4.10]{Ferenczi-lagrange}.

\bigskip

The common tool in the proofs of Theorems~\ref{thm:rec01} and~\ref{thm:lagrange}  
is Theorem~\ref{thm:delta_estimate_in_convex} concerning an unconventional packing problem 
in $\R^{2}$ (see Section \ref{sec:pack} for precise setting), which determines the relevant
optimal constant $\frac1{\sqrt 5}$. Section~\ref{sec:interval_recurrence} is
dedicated to the proof of Theorem~\ref{thm:rec01}.  And in
Section~\ref{sec:lagrange_iet} we provide the proof for
Theorem~\ref{thm:lagrange}.

\subsection{Further Comments}

\subsubsection*{The Classical Lagrange spectrum} 
There are many results about the Lagrange spectrum $\cL_1$ (of rotations), including the following.
\begin{enumerate}
\item $\cL_1$ starts with a discrete sequence $\sqrt{5}$, $2\sqrt{2}$, $\sqrt{221}/5$, \ldots with an accumulation point at $3$ \cite{Markov},
\item $\cL_1$ contains the half line $[c,+\infty)$ where $c = \frac{2221564096 + 283748 \sqrt{462}}{491993569} \simeq 4.528$, and
this half line is maximal (i.e. $\cL_1$ does not contains a half-line $[c',+\infty)$ with $c' < c$) \cite{Hall47}, \cite{Freiman75},
\item $\Hdim(\cL_1 \cap [0,t]) = 0$ if and only if $t \leq 3$ and $\Hdim(\cL_1 \cap [0, 2\sqrt{3}]) = 1$ \cite{Moreira}.
\end{enumerate}
The interval exchange Lagrange spectrum $\cL_d$ contains $m \cL_{d'}$ if $d = m d'$. In particular
$\cL_d$ contains $d \cL_1$, from which some properties follow (such as the existence of a half line).
But nothing as precise as the three above items for $\cL_1$ is known in general for $\cL_d$.

\subsubsection*{Lebesgue-preserving maps on subsets in $\R^{n}$ of finite volume}
Recall that quantitative Poincar\'e recurrence (almost everywhere) results 
are possible in more general settings of transformations of arbitrary 
metric spaces having finite Hausdorff dimension, 
see~\cite{Boshernitzan-recurrence}. In particular, 
the following result holds.
\begin{theorem}[\cite{Boshernitzan-recurrence}, Theorem~1.5]\label{thm:rec2}
Let $(X,d)$ be a metric space and let $\mu$ be a probability measure on it
which coincides, for some $\alpha\in(0,\infty)$, with the $\alpha$-Hausdorff 
measure on $(X,d)$.  
Then, for any transformation $T\colon X\to X$ which preserves the measure $\mu$, we
have
\[
\liminf_{n \to \infty} n^{1/\alpha} d(x, T^n x) \leq 1,
\qquad \text{for $\mu$-almost every $x \in X$.}
\]
\end{theorem}
Now, denote by $\mu_{d}(\cdot)$ the Lebesgue measure on $\R^{d}$ where $d\geq1$.
Let also $\rho$ be a norm on $\R^d$ and $B_\rho$ be the unit ball for this norm.
Let $X \subset \R^d$ be a measurable set of finite non-zero measure and let
$T\colon X \to X$ be a transformation which preserves $\mu_d$.
Then Theorem~\ref{thm:rec2} implies that for almost every $x \in X$ we have
\[
\liminf_{n \to \infty} \ n^{1/d}\rho({x} - T^n{x})\leq 
2\left(\frac{\mu_{d}(X)}{\mu_{d}(B_{\rho})}\right)^{1/d}.
\]
In the particular case when $X = [0,1)^d$ is the unit cube,
the above inequality takes the form
\[
\liminf_{n \to \infty} \ n^{1/d}\rho({\bf x} - T^n{\bf x})\leq 
2\ \mu_{d}(B_{\rho})^{-1/d}.
\]
In the even more special case where $X = [0,1)$ and $\rho$ is
the absolute value, we obtain a weaker version of Theorem~\ref{thm:rec01}
with the constant $1$ instead of $\frac1{\sqrt5}$. The technique we use in this
article to determine 
the optimal constant $\frac1{\sqrt5}$ for $X = [0,1)$ does not seem to extend to 
higher dimensions; not even for the square $X=[0,1)^2$.


\subsubsection*{Singular vectors and the Dirichlet spectrum}
We have seen one extension of the Lagrange constant of 1-dimensional rotation
to interval exchange transformations. It can also be defined for higher-dimensional
rotations as follows. Given $\alpha \in \R^d$ we define, similarly to~\eqref{eq:lagrange_rot_dim1},
\[
L(\alpha) = \limsup_{q \to \infty}\ \frac{1}{q^{1/d}\,\DZ{q\alpha}}.
\]
where $\DZ{x}$ denotes the Euclidean distance to the nearest integer lattice point. In
both contexts, the Lagrange constant also has a natural $\liminf$ counterpart
that we discuss next.

Given $\alpha\in \R^d$ we define its \emph{Dirichlet constant} $D(\alpha)$ as
\[
D(\alpha) = \liminf_{n \to \infty} \frac{1}{n^{1/d} \epsilon_n(\alpha)}
\qquad \text{where} \qquad
\epsilon_n(\alpha) = \min_{1 \leq k \leq n} \DZ{k \alpha}
\]
Recall that a vector $\alpha$ is called \emph{singular} if $D(\alpha) = +\infty$
(such vector only exists if $d \not= 1$). The set of singular vectors is known
to be of zero measure in any dimension, and its Hausdorff dimension has recently
been computed by N. Chevallier and Y. Cheung~\cite{ChevallierCheung}. Recall
that for the Lagrange constant, a vector $\alpha$ such that $L(\alpha) < +\infty$
is called \emph{badly approximable}. For a Lebesgue generic $\alpha$ we have
$D(\alpha) = c_d$ and $L(\alpha) = +\infty$, where $0 < c_d < +\infty$ is a
constant that only depends on the dimension.

Similarly, if $T$ is an i.e.t. that satisfies the Keane condition we define
its \emph{Dirichlet constant} as
\[
D(T) = \liminf_{n \to \infty} \frac{1}{n \epsilon_n(T)}.
\]
The \emph{Dirichlet spectrum}\footnote{
In~\cite{Ferenczi-lagrange} the Dirichlet spectrum is called the \emph{upper
Boshernitzan-Lagrange spectrum}. The reason for this is that M. Boshernitzan
proved that for i.e.t. the condition $D(T) < +\infty$ implies unique
ergodicity~\cite{Boshernitzan-uniqueergodicity}.}  is the set of possible 
Dirichlet constants for a given class of systems 
(i.e., a dimension for rotations, or a number of intervals for i.e.t.s, are fixed).

For rotations (or 2-i.e.t.s), the Dirichlet spectrum has a structure similar to
the Lagrange spectrum: that is, it starts with a discrete sequence and contains
an interval (see the discussion and references in the introduction
of~\cite{AkhunzhanovShatskov2013}). But the situation changes dramatically when
one goes to higher dimensional situations.  For instance, for both the
2-dimensional rotations~\cite[Theorem~1]{AkhunzhanovShatskov2013} and 3-i.e.t.s
\cite[Theorem~4.14]{Ferenczi-lagrange} the Dirichlet spectrum is an interval.
Nothing seems to be known about the structure of Dirichlet spectrum
for rotations in $\R^3$ or 4-i.e.t.s.

\section{An unconventional packing problem in $\R^{2}$}\label{sec:pack}
Denote by $\C$ the set of complex numbers. Given $z = x + iy$ we define
$\Norm(z) = \sqrt{|xy|}$ (and $\Norm^2(z) = (\Norm(z))^2 = |xy|$). The
function $\Norm$ can be thought as a
generalization of a norm whose unit ball is the region delimited by the
hyperbolas $xy = \pm 1$. Unlike with a genuine norm, the unit ball of $\Norm$, namely 
$\{z\colon \Norm(z) \leq 1\}$,
is not convex. However, it is still star shaped, and satisfies $\Norm(tz) = |t| \Norm(z)$ 
for all real $t$.

Given a polygon $P$ with vertices $z_1$, $z_2$, \ldots, $z_n$, we define
its $\Norm$-perimeter as $\displaystyle p(P) = \sum_{i=1,\ldots,n}\limits \Norm(z_{i+1} - z_i)$ (where indices are taken modulo $n$).
Our main tool in the present paper is given by the following result.

\begin{theorem}[\cite{Smith}] \label{thm:delta_estimate_in_convex}
Let $\Gamma$ be a finite set of points in $\C$ such that $N(x-y) \geq 1$
for every pair of distinct points $(x,y)$ of $\Gamma$.
Let $C$ be its convex hull. Let $A$ and $p$ be respectively the area and
$\Norm$-perimeter of $C$. Then
\[
\# \Gamma \leq \frac{1}{\sqrt{5}} A + \frac{p}{2} + 1.
\]
Moreover, if equality holds, then the set $\Gamma$ is a subset of a golden
lattice.
\end{theorem}

In the case of norms (i.e.,~when the unit ball is convex) the above result
was a conjecture of H.~Zassenhaus, which was proven in full generality by N.~Oler~\cite{Oler}.

\begin{theorem}[\cite{Oler}] \label{thm:delta_euclidean}
Let $\rho$ be a norm in $\R^2$ and let $\Gamma$ be a finite set of points
such that $\rho(x-y) \geq 1$ for all pairs $(x,y)$ of points of $\Gamma$.
Let $C$ be the convex hull of $\Gamma$, $A$ the area of $C$ and $p$ the
$\rho$-perimeter of $C$. Then
\[
\# \Gamma \leq \Delta(\rho) A + \frac{p}{2} + 1
\]
where $\Delta(\rho)$ is the critical determinant of the unit ball of $\rho$.
\end{theorem}
The \emph{critical determinant} $\Delta(C)$ of a centrally symmetric convex body
$C$ in $\C$ is defined as follows. A lattice $\Lambda \subset \C$ is said to be
$C$-admissible if $C \cap \Lambda = \{0\}$.  Then
\[
\Delta(C) := \min \{\det(\Lambda):\ \text{$\Lambda$ is $C$-admissible}\}.
\]
In Theorem~\ref{thm:delta_estimate_in_convex} the constant $\frac{1}{\sqrt{5}}$ is
also a critical determinant (for a star but non-convex body). This constant is achieved
exactly by the golden lattices. These two facts are just a reformulation of
Hurwitz Theorem (Theorem~\ref{thm:Hurwitz_rotation}).

The rest of this section is devoted to the proof of Theorem~\ref{thm:delta_estimate_in_convex}.
A complete proof is given in the  PhD thesis of N.~E.~Smith (a student of
H.~Zassenhaus), see~\cite{Smith}. Our proof uses the same path except that a delicate
induction is avoided by using Delaunay triangulations.

\begin{remark}
As pointed out in the review paper~\cite{Zassenhaus} a weaker version of
Theorem~\ref{thm:delta_estimate_in_convex} can be derived
from Theorem~\ref{thm:delta_euclidean} as shown in the PhD thesis
of Sr. M. R. von Wolff~\cite{vonWolff}. Namely, we always have
$N(z) \leq \rho(z)$ where $\rho(z) = (|x|+|y|)/2$. Hence
if $N(z) \geq 1$ then a fortiori $\rho(z) \geq 1$ and
Oler's result applies. Luckily the critical determinants are
the same for $N$ and $\rho$ equal to $1/\sqrt{5}$, though in the error term
the $\rho$-perimeter is generally larger than the $N$-perimeter.
Note that this weaker result would have been enough for our applications but we
prefer to include a self-contained and short proof of
Theorem~\ref{thm:delta_estimate_in_convex}.
\end{remark}

\begin{remark}
It would be tempting to conjecture that Oler's result actually holds
for centrally symmetric bodies. But this is actually false. Sr.\ M.\ R. von Wolff provided
a counterexample in her PhD thesis~\cite{vonWolff}.
\end{remark}

\subsection{Admissible triangles}
Given a triangle $(p,q,r)$ in $\C$, it is always inscribed in a smallest rectangle, 
namely the rectangle $R(p,q,r) = [x^-, x^+] \times [y^-,y^+]$ defined by
\[
\begin{array}{l@{\qquad}l}
x^- = \min(\Re(p),\Re(q),\Re(r)) & x^+ = \max(\Re(p),\Re(q),\Re(r)) \\
y^- = \min(\Im(p),\Im(q),\Im(r)) & y^+ = \max(\Im(p),\Im(q),\Im(r)).
\end{array}
\]

\begin{definition}
We call a triangle $(p,q,r)$ in $\C$ \emph{admissible} if the three points $p,q,r$ 
are on the boundary of the minimal rectangle $R(p,q,r)$ and no two of them are on the same side.
\end{definition}

\begin{remark}
One can alternatively define admissible triangles as triangles for which the sign of the slopes
of the sides are not all the same. This is the definition proposed in~\cite{Smith} on
page 7 in which admissible triangles are called \emph{type (a)}.
\end{remark}

Let $(p,q,r)$ be an admissible triangle. On the rectangle $R(p,q,r)$ exactly one vertex is in a corner. 
By convention we always label the sides $a,b,c$ so that the two sides $a,b$ 
are adjacent to that corner and $a$, $b$, $c$ are taken in counter-clockwise
order. (See Figure~\ref{fig:admissible_vs_nonadmissible} above.)
\begin{figure}[!ht]
\begin{center}
\includegraphics{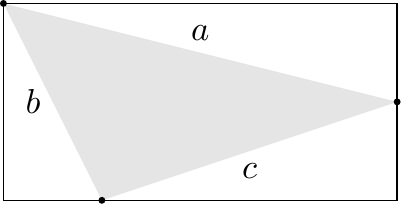} \hspace{1cm}
\includegraphics{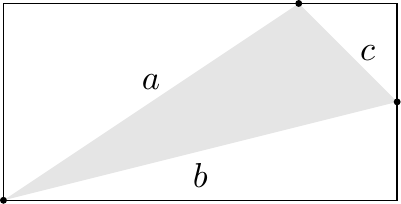} \hspace{1cm}
\includegraphics{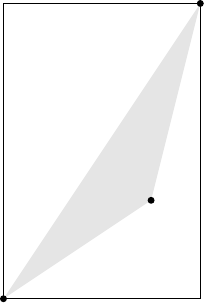}
\end{center}
\caption{Two admissible triangles and one non-admissible triangle.}
\label{fig:admissible_vs_nonadmissible}
\end{figure}

Recall that $\Norm\colon \C \to \R$ is the function defined by $\Norm(x+iy) = \sqrt{|xy|}$.
The main ingredient of the proof of Theorem~\ref{thm:delta_estimate_in_convex} is 
the following lemma which establishes the fact that the area of an
admissible triangle is completely determined by the $\Norm$-length of its sides.
\begin{lemma} \label{lem:formula_area_triangle}
Let $a,b,c$ be three sides of an admissible triangle 
$\Delta$ and let
\[
\alpha := \Norm(a)^2, \quad \beta := \Norm(b)^2, \quad \gamma := \Norm(c)^2.
\]
Then
\begin{equation}\label{eq:abcd}
\area(\Delta) = \frac{ \sqrt{\alpha^2 + \beta^2 + \gamma^2 - 2\alpha\beta + 
2\alpha \gamma + 2\beta \gamma}}{2}.
\end{equation}
\end{lemma}

\begin{proof}
Applying $g_t$ and an homothety of $1/\sqrt{\gamma}$ we can assume 
(up to symmetry) that $c=(1,-1)$ as shown 
in the picture below. 
\begin{center}
\includegraphics{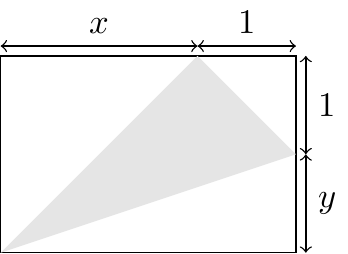}
\end{center}

Then we have \
$\alpha=x(y+1)$, 
$\beta=y(x+1)$, 
$\gamma=1$ \
while \ $\area(\Delta) = \frac{1+x+y}{2}$, and
the validation of the formula~\eqref{eq:abcd} becomes straightforward.
\end{proof}

\begin{corollary} \label{cor:area_triangle}
Let $\Delta$ be an admissible triangle with $a$, $b$, $c$, $\alpha$, $\beta$, $\gamma$
as in Lemma \ref{lem:formula_area_triangle}.
Let $m = \min \{\alpha,\beta,\gamma\}$ and $M = \max \{\alpha,\beta,\gamma\}$.
Then
\[
\frac{\sqrt{5}}{2} \ m \leq \area(\Delta) \leq \frac{\sqrt{m^2 + 4M^2}}{2} \leq \frac{\sqrt{5}}{2} \ M.
\]
If moreover $\area(\Delta) \leq \sqrt{2}\ m$ then
\[
\frac{\sqrt{M^2 + 4m^2}}{2} \leq \area(\Delta).
\]
\end{corollary}

\begin{proof}
Let us set $m=\min(\alpha,\beta,\gamma)$ and $M=\max(\alpha,\beta,\gamma)$.
By symmetry we can assume that $\alpha \geq \beta$.
Let $f(\alpha,\beta,\gamma) = \alpha^2 + \beta^2 + \gamma^2 - 2 \alpha\beta + 2\alpha\gamma + 2\beta\gamma$.
From Lemma~\ref{lem:formula_area_triangle} we have $\area(\Delta) = \sqrt{f(\alpha,\beta,\gamma)}/2$.

As $\alpha^2 + \beta^2 \geq 2\alpha\beta$, we have
\[
f(\alpha,\beta,\gamma)
\geq \gamma^2 + 2\alpha\gamma + 2\beta\gamma
\geq 5 m^2
\]
which proves the lower bound. For the upper bound, if $M=\alpha$ then we can use the fact that $-2 \alpha \beta + 2 \beta \gamma \leq 0$. If $M = \gamma$ then we use
\[
f(\alpha,\beta,\gamma)
= (\alpha-\beta)^2 + \gamma^2 + 2\alpha\gamma + 2\beta\gamma
\leq (\gamma-\beta)^2 + \gamma^2 + 2\alpha\gamma + 2\beta\gamma
\leq 4M^2 + m^2.
\]

To prove the last statement we analyze the function $f(\alpha,\beta,\gamma) = \alpha^2 + \beta^2 + \gamma^2 - 2 \alpha\beta + 2\alpha\gamma + 2\beta\gamma$. For each possibility of maximum and minimum, we just analyze $f$ as a one-variable function. The values of the extrema can be computed by elementary calculus. We summarize this information in the following array.
\[
\begin{array}{l|cccc}
\text{order on $\alpha,\beta,\gamma$}                 & \argmin(f)           & \min(f)                 & \argmax(f)             & \max(f)             \\ \hline
m = \beta\ \text{and}\ M = \gamma        & \alpha=m   & 4Mm + M^2   & \alpha=M & m^2 + 4 M^2 \\ \hline
m = \beta\  \text{and}\ M = \alpha       & \gamma=m   & M^2 + 4 m^2 & \gamma=M & 4 M^2 + m^2 \\ \hline
m = \gamma\ \text{and}\ M = \alpha       & & & & \\
 \alpha/\gamma \leq 2        & \beta=m   & M^2 + 4m^2 & \beta=M & m^2 + 4M m \\
 2 \leq \alpha/\gamma \leq 3 & \beta=M-m & 4M m       & \beta=M & m^2 + 4Mm \\
 3 \leq \alpha /\gamma       & \beta=M-m & 4M m       & \beta=m & M^2 + 4m^2 \\
\hline
\end{array}
\]
In the two first cases $m=\beta, M=\gamma$ or $m=\beta, M=\alpha$ we have the lower bound $f(\alpha,\beta,\gamma) \geq 4m^2+M^2$.
In the case $m=\gamma, M=\alpha$, the condition $\area(\Delta) \leq \sqrt{2}\ m$ implies that $\alpha/\gamma \leq 2$. Indeed,
if we had $M/m > 2$ then $\area{\Delta} \geq \sqrt{4Mm}/2 > \sqrt{2}\ m$. And
in the case $\alpha/\gamma \leq 2$, the lower bound $M^2 + 4m^2$ is valid.
\end{proof}

Note that the gap between $\sqrt{5}/2 \simeq 1.118$ and $\sqrt{2} \simeq 1.4142$ is not large. But having this gap is essential as it will allow us to get lower bounds from upper bounds in Section~\ref{sec:lagrange_iet} via the following lemma.
\begin{corollary} \label{cor:distorsion_from_area}
Let $\Delta$ be an admissible triangle with $a$, $b$, $c$, $\alpha$, $\beta$, $\gamma$
as in Lemma \ref{lem:formula_area_triangle}.
Let $m = \min \{\alpha,\beta,\gamma\}$ and $M = \max \{\alpha,\beta,\gamma\}$.
Assume that
$\area(\Delta) \leq \left(\frac{\sqrt{5}}{2} + \epsilon\right)m$,
for some\, $\epsilon$, $0 < \epsilon < \sqrt{2} - \frac{\sqrt{5}}2$.
Then $M \leq (1 + (2\sqrt{2} + \sqrt{5})\epsilon) m$.
\end{corollary}

\begin{proof}
Because $\epsilon < \sqrt{2} - \frac{\sqrt{5}}2$ the second half of Corollary~\ref{cor:area_triangle} holds: $M^2 + 4m^2 \leq 4 \left(\area(\Delta)\right)^2$. 
Using the hypothesis, we get
$M^2 + 4m^2 \leq 4 ( \frac{\sqrt{5}}2 + \epsilon)^2 m^2$ and hence 
$M^2 \leq (1 + 4\sqrt{5} \epsilon + 4\epsilon^2) m^2 < 
(1 + 4(\sqrt{2}+\frac{\sqrt{5}}2)\epsilon) m^2$. Taking square roots in this last inequality and applying the inequality $\sqrt{1+x} \leq 1+x/2$, which is valid for all $x > 0$, we get the result.
\end{proof}

\subsection{$L^\infty$-Delaunay triangulations}
For a general reference on Delaunay triangulations we refer the reader to~\cite{Okabe-tessellations}.

\begin{wrapfigure}{R}{0.45\textwidth}
\centerline{\includegraphics{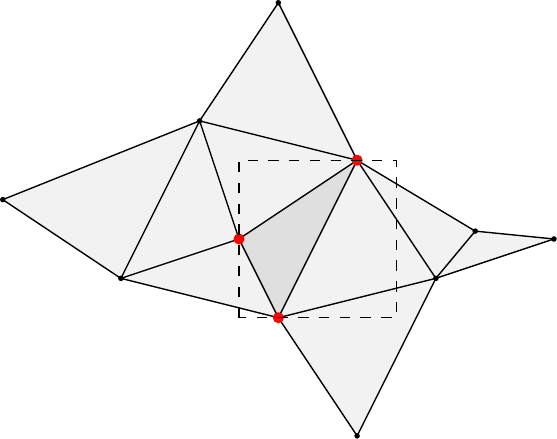}}
\caption{An $L^\infty$-Delaunay triangulation with one of the maximal squares $S$.}
\label{fig:delaunay}
\end{wrapfigure}
Let $\Gamma \subset \C$ be a finite set of points. A \emph{triangulation} of
$\Gamma$ is a set of triangles with disjoint interiors whose vertex set is
contained in $\Gamma$. Note that we have no maximality assumption here.
The $L^\infty$-\emph{Delaunay triangulation} of $\Gamma$ is defined as follows: 
a triangle with vertices $p,q,r \in \Gamma$ belongs to that triangulation 
if and only if there exists a square $S \subset \C$ with horizontal and vertical sides
such that $S \cap \Gamma = (\partial S) \cap \Gamma = \{p,q,r\}$. An example of a
Delaunay triangulation is provided in Figure~\ref{fig:delaunay}.

In some cases, there might be more than three points on the boundary of a
square. We will implicitly exclude the case where two points $z$ and $z'$
of $\Gamma$ are on the same horizontal or vertical line, as these
correspond to $\Norm(z - z') = 0$. Assuming that
$\displaystyle \min_{\stackrel{z,z' \in \Gamma}{z \not= z'}} \Norm(\Gamma) > 0$,
there are either three or four points on maximal squares. In the latter case, there is an
ambiguity as there are two different ways of making two triangles out of these
four points. We will abuse the terminology and still speak about \emph{the}
Delaunay triangulation for one of the triangulations obtained after making
a choice in each quadruple of points in a maximal square.

\begin{lemma} \label{lem:Linfinity_Delaunay_properties}
Let $\Gamma \subset \C$ be a finite set that contains at least three
points and is such that no pair of points are on the same horizontal or vertical
line. Let:
\begin{itemize}
\item[\em(a)] $C$ be the convex hull of $\Gamma$;
\item[\em(b)] $T$ be the finite collection of closed triangles determined by
 the $L^\infty$-Delaunay triangulation of $\Gamma$ (the interiors
 of these triangles are disjoint);
\item[\em(c)] $U=\bigcup_{\Delta\in T}\Delta$ be the union of all these triangles.
\end{itemize}
 Then the following statements hold.
\begin{itemize}
\item[\bf1.] The $L^\infty$-Delaunay triangulation $T$ contains only admissible triangles.
\item[\bf2.] The set $U$ is simply connected.
\item[\bf3.] The $\Norm$-length of $\partial U$ is smaller than the $\Norm$-length of $\partial C$
(where $\partial U$ and $\partial C$ stand for the boundaries of $U$ and $C$, respectively).
\end{itemize}
\end{lemma}
\vspace{1pt}
\begin{proof}
The first statement is immediate from the definition. Indeed, each triangle of $T$
is inscribed in a square with its three vertices on the sides (by definition of
the Delaunay triangulation) and since no pair of points of $\Gamma$ are on the same
horizontal or vertical line, they belong to different sides.

The segments $[p,q]$ on the boundary $\partial U$ of $U$ are the ones such that there
exist arbitrary large squares $S$ with $S \cap \Gamma = \{p,q\}$. Such a segment
needs to be on the boundary.

For the third statement, let $\gamma=[p,q]$ be an edge of the convex hull $C$
of $\Gamma$ that is not an edge of a triangle in $T$. (Thus $\gamma\in\partial C$ but 
$\gamma\notin \partial U$).
The line through $\gamma$ separates the plane into two regions, and one of them 
contains all points of $\Gamma$ except $p$ and $q$.
Without loss of generality we assume that $\Re(p) < \Re(q)$, $\Im(p) < \Im(q)$ and that
the points of $\Gamma$ are above the line through $p$ and $q$ as in the following picture.
\begin{center}
\includegraphics{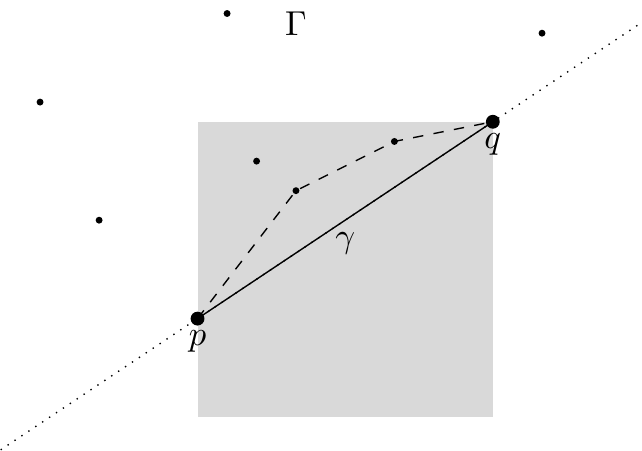}
\end{center}
The segment $\gamma$ is not an edge of a triangle in $T$ if and only if there are points 
in the square $S$ which admits $\gamma$ as a chord. 

Let $\Gamma_0$ be the set 
of points in the interior of $S$ that are different from $p$ and $q$.
Let $p_1$ be the point in $\Gamma_0$ with lowest imaginary part.
Now we proceed inductively until \ $\Gamma_n$ \ by defining
\[
\Gamma_n := \{x \in \Gamma_{n-1} \mid \Re(x) > \Re(p_n) \quad \text{and} \quad 
\Im(x) > \Im(p_n)\}.
\]
and, if $\Gamma_n\neq\emptyset$, we continue with picking\,  $p_{n+1}\in \Gamma_n$ with 
the lowest imaginary part.

Let  $p_1,\ldots,p_n$ be the points selected in the above way when the process stops
(i.e., when $\Gamma_n=\emptyset$). By adding two more points $p_{0}=p$ and $p_{n+1}=q$, 
we end up with the $(n+2)$ points
$p_0=p,p_1,\ldots,p_n,p_{n+1}=q$,  with the edges  $[p_{k},p_{k+1}]$, 
$0\leq k\leq n$, forming the contour $\phi_{p,q}$ of $\partial U$ 
between $p$ and $q$. 

Now, we claim that $\Norm(q-p) \geq \sum_{i=0}^n \Norm(p_{i+1} - p_i)$. This is to say
that the triangle inequality is actually reversed! It follows from the fact
that $\Norm$ restricted to the positive quadrant is concave.

This completes the proof of Lemma \ref{lem:Linfinity_Delaunay_properties}.
\end{proof}

\subsection{Proof of Theorem~\ref{thm:delta_estimate_in_convex}}\label{sec:proofpack}

Let $\Gamma \subset C$ be a finite set of cardinality $s$ and $C$ its convex hull.
Let $T$ be the $L^\infty$-Delaunay triangulation of $\Gamma$ and let 
$U$ be the union of the closed triangles in $T$ (notations just as in 
Lemma~\ref{lem:Linfinity_Delaunay_properties}). Next, we establish 
a lower bound (see \eqref{eq:bound-on-n}) on the number $n$ of triangles in $T$.

The set $\Gamma$  can be partitioned into the 
three subsets  $\Gamma=\Gamma_2\cup\Gamma_{bad}\cup\Gamma_{good}$
as follows:
\begin{enumerate}
\item The set $\Gamma_{2}$ of $2$ \emph{special} points that lie on the extreme left and extreme right of $\Gamma$,
\item The set $\Gamma_{bad}=(\Gamma\cap\partial U)\setminus\Gamma_2$ of 
$s_{bad}$ points that lie on $\partial U$ but not in 
$\Gamma_2$,
\item The set $\Gamma_{good}=\Gamma\setminus(\Gamma_2\cup\Gamma_{bad})$ of remaining $s_{good}=s-2-s_{bad}$ points that lie in the interior of $U$.
\end{enumerate}

Since the $\Norm$-distance between any two points of $\Gamma$ is at least 1 we
have that $s_{bad} + 2$ is smaller than the $\Norm$-perimeter of $U$. But
from Lemma~\ref{lem:Linfinity_Delaunay_properties} we know that the $\Norm$-perimeter
of $U$ is actually smaller than that of $C$. Hence $s_{bad} + 2 \leq p$.

Next, with each triangle $\Delta$ of $T$, we associate a point in $\Gamma$ as follows. There is
exactly one vertex of $\Delta$ for which the vertical line through that point 
intersects the interior of the triangle as in the following pictures:
\begin{center}
\includegraphics{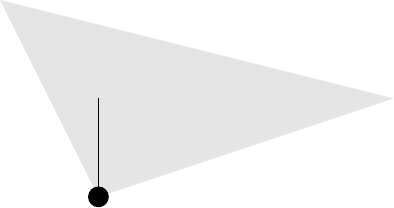} \hspace{1cm} \includegraphics{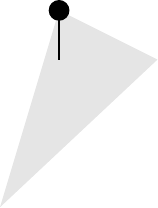}
\end{center}
It is easy to see that
\begin{enumerate}
\item for the (two) points in $\Gamma_{2}$, there are no associated triangles,
\item for each of the $s_{bad}$ points in $\Gamma_{bad}$ there is exactly one associated triangle,
\item for each of the $s_{good}$ points in $\Gamma_{good}$ there are exactly two associated triangles.
\end{enumerate}

In other words, the number of triangles in $T$ is given by the formula $n = 2
s_{good} + s_{bad}$. Now, substituting $s_{good} = s-2-s_{bad}$ and taking into
account the inequality $s_{bad}+2 \leq p$ we obtain
\begin{equation}\label{eq:bound-on-n}
n = 2 s_{good} + s_{bad} = 2s - s_{bad} - 4 \geq 2s - p - 2
\end{equation}
By Corollary~\ref{cor:area_triangle}, each of the triangles in $T$ has area 
at least $\frac{\sqrt{5}}{2}$. So
\[
A \geq \frac{\sqrt{5}}{2}\, n \geq \frac{\sqrt{5}}{2} \left(2s - p - 2 \right),
\]
and the inequality claimed in Theorem~\ref{thm:delta_estimate_in_convex} follows.

\subsection{The rectangular case}
The special case where $C\subset \R^{2}=\C$ is a rectangle with sides parallel to the 
coordinate axes is addressed by the following result. It is in this form that our packing
result will be used in Section~\ref{sec:interval_recurrence}.

\begin{theorem} \label{thm:rectangular}
Let $C=[x^-,x^+]\times[y^-,y^+]\subset\R^{2}$ be a rectangle of area 
$A = (x^+ - x^-)(y^+-y^-)$, and let $\Gamma \subset C$ be a finite subset
of cardinality $s\geq2$.
Set \ $\delta = \min_{\substack{x,y\in C\\ x\neq y}}\limits\, \Norm(x-y)$ and $A' = A/\delta^2$. Then 
\vspace{-2mm}
\begin{equation}\label{eq:rect1}
s\leq \frac{A'}{\sqrt 5}+ \sqrt{2 A'} + 1.
\end{equation}
In particular, for a given $\epsilon>0$,  we have
\begin{equation}\label{eq:rect2}
s\leq\left(\frac 1{\sqrt 5}+\epsilon\right)\cdot A'
\end{equation}
provided that either $A'$ or $s$  are large enough.
More precisely, for\, $0<\epsilon<1/2$, each of the following two conditions 
\begin{equation}\label{eq:twoconditions}
\text{either \quad {\bf(c1)}}\ A'>\tfrac4{\epsilon^2},\hspace{12mm} \text{or \quad {\bf(c2)}}\ 
s>(\tfrac3\epsilon+1)^2,
\end{equation}
suffices for the inequality \eqref{eq:rect2} to hold.
\end{theorem}

\begin{proof}
We show how to deduce Theorem~\ref{thm:rectangular} from
Theorem~\ref{thm:delta_estimate_in_convex}.
We adopt the notation used in these two results.

By Theorem \ref{thm:delta_estimate_in_convex}, we have 
$
s \leq \frac A{\delta^2 \sqrt 5}+\frac{p}{2\delta}+1
$
where $p$ is the $\Norm$-perimeter of the convex hull of $\Gamma$.
As all quantities in the above inequality are invariant by the linear
action of the diagonal flow $x+iy \mapsto e^t x + i e^{-t} y$ we can
assume that $C$ is a square with side length $\sqrt{A}$. The $\Norm$-perimeter
is always smaller than $\sqrt{2}/2$ times the euclidean perimeter (since the 
euclidean ball of radius $\sqrt{2}$ is contained in the $N$-ball of radius
$1$). Moreover, if $C_1 \subset C_2$ are two convexes, it is well known that
the euclidean perimeter of $C_1$ is smaller than the one of $C_2$.  Hence $p
\leq 2\sqrt{2}\, \sqrt{A}$. The equation~\eqref{eq:rect1} follows.

It remains to check the inequality $\frac{\sqrt{2A}+1}A<\epsilon$, assuming
that $0<\epsilon<1$ and that at least one of the conditions in
\eqref{eq:twoconditions}, either ({\bf c1}) or ({\bf c2}),  holds.

The condition {\bf(c1)}, $A>\frac4{\epsilon^2}$, implies that $\frac1{\sqrt A}< \frac1{2}$, and
then $\frac{\sqrt{2A}+1}A< \frac{2}{\sqrt A}<\epsilon$.

The condition {\bf(c2)}, $s>(\frac2\epsilon+1)^2$, implies that $s\geq25$, 
and then~\eqref{eq:rect1} implies
that $s\leq(\sqrt A+1)^2$. We obtain $\sqrt A\geq\sqrt s-1>\frac2\epsilon$,
and then $A>\frac4{\epsilon^2}$. Thus
{\bf(c2)} implies {\bf(c1)}, the case already established.

This completes the deduction of Theorem~\ref{thm:rectangular}  from Theorem~\ref{thm:delta_estimate_in_convex}.
\end{proof}

\section{Recurrence in the interval} \label{sec:interval_recurrence}
This section is dedicated to the proof of Theorem~\ref{thm:rec01}. In the first
part we prove a technical step involving an estimation
of the measure of points with a lower bound on the rate of recurrence. This proposition uses
the unconventional packing result of Theorem~\ref{thm:delta_estimate_in_convex} (in
its form given in Theorem~\ref{thm:rectangular}) and basic ergodic theory. In a second part we derive
Theorem~\ref{thm:rec01} using the Lebesgue density theorem.

\subsection{An estimate for the measure of badly recurrent points}
The following estimate is used in the proof of Theorem~\ref{thm:rec01}.
\begin{proposition}\label{prop.Vr}
Let \ $T\colon [0,1) \rightarrow [0,1)$ be a measurable map which 
preserves the Lebesgue measure $\mu$.  \\
Let \ $V\subset [0,1)$ be a non-empty open subinterval.
Set
\[
\rho(x,V)=\inf_{\substack{n\geq1 \\[.4mm] T^nx\in V}} \ n\cdot|T^n(x)-x|\quad
     \text{\em(for } \ x\in V),
\]
and for $r > 0$ define the subsets  $V_r=\{x\in V\mid \rho(x,V)\geq r\}\subset V$.
Then $\displaystyle \mu(V_r)\leq\frac{\mu(V)}{r\sqrt 5}$.
\end{proposition}
The above proposition is trivial for $r\leq\frac1{\sqrt 5}$
(because then it follows immediately from the inclusion $V_r\subset V$).
Note also that one can recover the recurrence rate $r(T,x)$ from $\rho(x,V)$ as
\[
r(T,x) = \liminf_{\epsilon_1,\epsilon_2 \to 0} \rho(x, (x-\epsilon_1,x+\epsilon_2)).
\]

\begin{proof}
Let $S\colon V\to V$  be the first return map on $V$ induced by $T$.
Thus
\[
S(j)=T^{F(x)}(x) \quad \text{(for a.a. $x\in V$)},
\]
where $F(x)$ is the minimal integer $n\geq1$ such that $T^n(x)\in V$.
By Kac's lemma, the function $F\colon V\to\N$  is defined  almost everywhere and
$\int_{V} F(x)\,dx\leq1$. For $n\geq0$, we have
\[
S^n (x)=T^{F_n(x)}(x), \quad \text{where } 
  \  F_n(x)=\sum_{k=0}^{n-1} F(S^k x)   \quad \text{(for a.a. $x\in V$}).
\] 
By the Birkhoff Ergodic Theorem, the pointwise convergence
\begin{equation}\label{eq:individ1}
\lim_{n\to\infty}\frac{F_n(x)}n=G(x) \quad \text{(for a.a. $x\in V$)},
\end{equation}
to some integrable function $G(x)\geq1$ takes place, and
\begin{equation}\label{eq:integralG1}
\int_{V}G(x)\,dx=\int_{V} F(x)\,dx\leq1.
\end{equation}

Denote by $V'$ the subset of those $x\in V$ for which all the values $S^n(x)$,
$F_n(x)$ for all $n \geq 0$ and $G(x)$ are defined and~\eqref{eq:individ1}
holds. Clearly, $\mu(V')=\mu(V)$.

Fix $x\in V'$ and let $\{z_{n}\}_{n\geq0} = \{x_n+ iy_n\}_{n\geq 0}$ be 
the sequence in $\C$ defined by the formula
\begin{align*}
x_n&=\Re(z_n)=F_n(x)\in\N\cup\{0\}, \\
y_n&=\Im(z_n)=T^{x_n}(x)=T^{F_n(x)}(x)=S^n(x)\in V.
\end{align*}
(Thus $\{x_n\}$ is a strictly increasing sequence of integers, 
with $x_0=0$, $z_0=it$). For an integer $q\geq 1$, let 
\begin{align*}
\Gamma_q=\Gamma_q(x)&=\{z_n\mid n\in[0,q-1]\}=\{z_0,z_1,\ldots,z_{q-1}\},\\
\Gamma'_q=\Gamma'_q(x)&=\big\{z_n\mid n\in[0,q-1], \ \text{\small and }  \Norm^{2}(z_n-z_m)\geq r, \ 
    \forall m\in [n+1,q-1]\}\big\} \subset \Gamma_q,\\
\Gamma''_q=\Gamma''_q(x)&=\Gamma_q\setminus\Gamma'_q=\{z_n\mid n\in[0,q-2], \ \text{\small and }
         \exists m\in [n+1,q-1] \ \text{\small \ such that }  \Norm^{2}(z_n-z_m)<r\}.
\end{align*}
Let also $C_q(x)=[0,x_q] \times V$ and $A_q(x)=\area(C_q(x))=\mu(V)\,x_q$. Given $\Gamma \subset \C$ as
the one defined above, we set
\begin{equation} \label{eq:delta_definition}
\delta(\Gamma) = \min_{z,z' \in \Gamma} \Norm (z - z').
\end{equation}

Next we establish the following inequality:
\begin{equation}\label{eq:ineqGpq}
M_x\colon\!=\limsup_{q\to\infty}\limits \frac {|\Gamma'_q(x)|}q\leq\frac{\mu(V)\,G(x)}{r\sqrt5} \quad \ \text{(for $x\in V'$).}
\end{equation}
Let us fix $x\in V'$ and chose an increasing sequence of positive integers $\{q_n\}_{n \geq 0}$ so that
\begin{equation}\label{eq:mt}
M_x=\lim_{n\to\infty}\limits \frac {|\Gamma'_{q_n}|}{q_n}.
\end{equation}
Let us also fix $\epsilon>0$. Then the inequality
\[
|\Gamma'_{q_n}|\leq \frac{A_{q_n}}{\delta^{2}(\Gamma'_{q_n})}\left(\frac1{\sqrt5}+\epsilon\right)
\]
holds for all large $n$, in view of \eqref{eq:rect2} in Theorem~\ref{thm:rectangular} and where $\delta$ is defined
by~\eqref{eq:delta_definition}.
(Note that $\Gamma'_q\subset C_q$ and $\lim_{n\to\infty}\limits|\Gamma'_{q_n}|=\infty$ because otherwise $M_x=0$ and~\eqref{eq:ineqGpq} becomes trivial).

Since $\delta^{2}(\Gamma'_q)\geq r$ (in view of the definition of $\Gamma'_q$) and $\epsilon>0$ is arbitrary, we get \
$
 \limsup_{n\to\infty}\limits \ \frac {|\Gamma'_{q_n}|}{A_{q_n}}\leq 
  \frac 1{r\sqrt5}
$
 \ and hence \
$
\limsup_{n\to\infty}\limits \frac {|\Gamma'_{q_n}|}{x_{q_n}}\leq \frac{\mu(V)}{r\sqrt5}
$
 \ (as $A_q=\mu(V) x_q$). 
Taking in account \eqref{eq:mt} and that \
$\lim_{q\to\infty}\limits\frac{x_q}q=\lim_{q\to\infty}\limits\frac{F_q(x)}q=G(x)$,  \
we obtain
$
M_x=\limsup_{q\to\infty}\limits \frac {|\Gamma'_{q_n}|}{q_n}\leq\frac{\mu(V)\,G(x)}{r\sqrt5},
$
and the inequality \eqref{eq:ineqGpq} follows.

Next we observe the inclusions $V_r\cap \Gamma_q(x)\subset \Gamma'_q(x)$ 
(see the definition of $V_r$ in Proposition~\ref{prop.Vr}) and conclude that
\begin{equation}\label{eq:1vdsum}
\sum_{n=0}^{q-1}\limits 1_{V_r}(S^n(x))=|V_r\cap \Gamma_q(x)|\leq|\Gamma'_q(x)|
\quad \ \text{(for $q\geq2$)}
\end{equation}
where $1_{V_r}$ stands for the characteristic function of the set $V_r$. 
Since the map  $S\colon V'\to V'$ is measure preserving, integrating~\eqref{eq:1vdsum} results in the inequality
\[
q\mu(V_r)=\int_{V'}\Big(\sum_{n=0}^{q-1} 1_{V_r}(S^n(x))\Big)\,dx
  \leq \int_{V'}|\Gamma'_q(x)|\,dx,
\]
whence
\[
\mu(V_r) \leq \int_{V'}\frac{|\Gamma'_q(x)|}q\,dx \leq 
    \int_{V'}\Big(\sup_{p\geq q}\frac{|\Gamma'_p(x)|}p\Big)\,dx.
\]
Passing to the limit $q\to\infty$, we get
\[
\mu(V_r)  \leq \int_{V'}\!\Big(\limsup_{q\to\infty}\frac{|\Gamma'_q(x)|}q\Big)\,dx \leq
   \int_{V'} \frac{\mu(V)\,G(x)}{r\sqrt 5}\,dx\leq\frac{\mu(V)}{r\sqrt 5},
\]
in view of \eqref{eq:integralG1} and \eqref{eq:ineqGpq}.
This completes the proof of Proposition \ref{prop.Vr}.
\end{proof}

\subsection{Derivation of Theorem \ref{thm:rec01} (from Proposition  \ref{prop.Vr})}
\label{sec:derivation}
\begin{proof}[Proof of Theorem~\ref{thm:rec01}]
Let $r>\frac1{\sqrt5}$ be given. Let $\Phi_N$ 
be the finite collections of subintervals of $[0,1]$ defined as below
\[
\Phi_N=\big\{[\tfrac nN, \tfrac {n+1}N]\mid 0\leq n\leq N-1\big\} \quad (\text{for }N\geq2).
\]
Each $\Phi_N$  partitions
the unit interval $[0,1]$ into $N$ subintervals of equal lengths $\tfrac1N$ (up to their boundaries).

Recall that for every subinterval $V\subset [0,1]$, the inequality
$\mu(V_r)\leq\frac {\mu(V)}{r\sqrt 5}$ \
holds (by Proposition \ref{prop.Vr}). In particular, 
\[
\mu(V_r)\leq\tfrac 1{Nr\sqrt 5}, \quad \text{for any } V\in \Phi_N.
\]

Next we introduce two sequences of sets
\[
W_N=\!\!\bigcup_{V\in\Phi_N}\limits \!\!V_r; \qquad  W'_N=\bigcap_{n\geq N}\limits \! W_n
\qquad \text{(for } \ N\geq2),
\]
and two additional sets
\begin{align}
W=&\bigcup_{N\geq2}W'_N= \liminf_{N\to\infty} \ W_N= 
    \big\{x\in[0,1]\ \big|\  x\in W_N, \ \text{\small for all $N $
    large enough}\big\},\label{eq:W}\\
U=&\ [0,1]\!\setminus\! W= \{x\in[0,1] \mid x\notin W_N, \ 
    \text{\small for infinitely many $N$} \big\}.\label{eq:U}
\end{align}
Every subinterval  $J\subset[0,1]$  can be covered by $[\mu(J)N]+2$ subintervals 
from the collection $\Phi_N$, so
\[
\limsup_{N\to\infty}\ \mu(W_N\cap J) \leq \limsup_{N\to\infty}
   \big(\tfrac{[\mu(J)N]+2}{Nr\sqrt 5}\big)
   =\tfrac1{r\sqrt 5}\cdot\mu(J),
\]
and hence, for every $N\geq2$,
\[
\mu(W'_N\cap J) \leq 
   \liminf_{n\to\infty}\,\mu(W_n\cap J)\leq \tfrac1{r\sqrt 5}\cdot\mu(J).
\]
Since  $J\subset[0,1]$ is arbitrary and $\tfrac1{d\sqrt 5}<1$, the sets  $W'_N$ cannot have Lebesgue density 
points, thus  $\mu(W'_N)=0$, and hence $\mu(W)=0$.
It follows that  $\mu(U)=\mu([0,1]\!\setminus\!W)=1$. 

Next fix $x\in U$ and set $H_x=\{n\geq2\mid x\notin W_n\}\subset\N$. By definition
of $U$ (see \eqref{eq:U}), the set $H_x$ is infinite.
Fix $N\in H_x$. Observe that \ $x\in V$, 
for some $V\in\Phi_N$  (as $\bigcup_{V\in\Phi_N}\limits \!\!V=[0,1]$), 
while \ $x\notin V_r$ \ (as $x\notin W_N$). It follows that
\[
\rho(x,V)=\inf_{\substack{n\geq1 \\[.4mm] T^nx\in V}}  n\cdot|T^n(x)-x|<r.
\]
We obtain (in view of the implication \ 
$\{x,T^nx\}\subset V\in \Phi_{N}\!\implies\! |T^nx-x|\leq\frac1N$) \ that
\[
\inf_{\substack{n\geq1 \\[.4mm] |T^nx-x|\leq\frac1N}}\!\!  n\cdot|T^n(x)-x|< r, \qquad 
   \text{for all } \ N\in H_x.
\]
Taking in account that $H_x$ is infinite, one concludes that $r(T,x)=\liminf_{n\to\infty} \ n\,|T^n(x)-x|\leq r$.
As the selection $x \in U$ and $r>\frac1{\sqrt5}$ are arbitrary and $\mu(U)=1$, the proof
of Theorem~\ref{thm:rec01}  follows.
\end{proof}

\section{Lagrange constants of interval exchange transformations}
\label{sec:lagrange_iet}
Our proof of Theorem~\ref{thm:lagrange} uses translation surfaces that can be thought
as suspension flows of interval exchange transformations. The Lagrange spectrum
of an interval exchange transformation will now be studied through the
$\Norm$-norm of edges that are part of some specific triangulation.

\subsection{From interval exchange transformations to translation surfaces}
%
%
First, we define translation surfaces. For more details, we invite the reader to
consult the survey by Masur and Tabachnilov~\cite{MasurTabachnikov}.

A \emph{translation surface} is a surface that is obtained from gluing 
$2d$ euclidean triangles by identifying their edges by translation. The simplest example
are tori $\C / \Lambda$ where $\Lambda$ is a lattice. The torus $\C / \Lambda$
with $\Lambda = \Z u \oplus \Z v$ is obtained by gluing together the two triangles
with sides respectively $(u, -u+v, -v)$ and $(v, -u, u-v)$. In that case we have $d=1$.
(See Figure~\ref{fig:torus}).
\begin{figure}[!ht]
\begin{center}\includegraphics{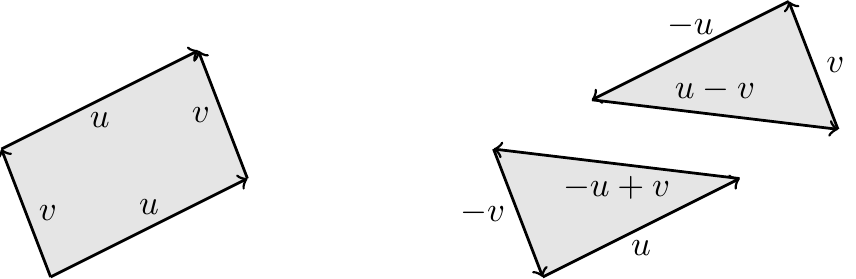}\end{center}
\caption{The fundamental domain of the torus seen as the gluing of two
euclidean triangles by translations.}
\label{fig:torus}
\end{figure}
Equivalently, a translation surface is a compact surface $X$ endowed with an
atlas defined on $X$ minus a finite (non-empty) set of points $\Sigma \subset
X$ with values in $\C$ such that coordinate changes are translations. (We assume
that the atlas is maximal for sake of uniqueness). It is easy to see that a
surface glued from euclidean triangles has such a geometric structure. Conversely,
on a translation surface there always exists a triangulation whose edges are flat
segments, and therefore the two definitions are equivalent.

Two translation surfaces $(X,\Sigma)$ and $(X',\Sigma')$ are isomorphic if
there exists a homeomorphism $\phi:X \to X'$ such that $\phi(\Sigma) = \Sigma'$
and for every chart $g: U \to \C$ of $X'$, $g \circ \phi$ is a chart of $f$. If the two
surfaces are given by a triangulation, the surfaces are isomorphic if and only if
one can pass from one to the other using edge flips (that is, we are
allowed to paste two triangles together and cut the resulting quadrilateral
along the other diagonal).

A point in the surface that is not a vertex of a triangle is
called a \emph{regular point}. The (image of the) vertices in the surface are called
the \emph{singularities}. As there are identifications, there can be fewer singularities
in the surface than vertices. Around each singularity, there is a well defined angle that
is a multiple of $2\pi$. If a translation surface $X$ is made of $2d$ triangles
then the set of conical angles of the singularities $2k_1 \pi, 2k_2 \pi,
\ldots, 2k_n \pi$ satisfies the relation $d = k_1+k_2+\ldots+k_n$. In the torus
of Figure~\ref{fig:torus}, we have $d=1$, $n=1$ and $k_1=2\pi$.

A translation surface inherits a translation structure: that is, given a point $p$ on the surface
and a vector $v \in \C$ one can define $p + v$ on the surface unless the segment
$p + tv$ contains a vertex for some $t$ with $0 \leq t < 1$. In general, we do not have $p + (v+w) = (p + v) + w$.
The \emph{(vertical) translation flow} on $X$ is the flow defined (almost everywhere) by
$\phi_t(p) = p + t \sqrt{-1}$.

A \emph{saddle connection} in a translation surface is a straight-line segment
that joins two singularities.

An horizontal segment $I$ in $X$ is \emph{admissible} if the orbits of its left and 
right extremities under the translation flow have the following property: either
in the past or in the future, the orbit hits a singularity before returning to
the interval.

\begin{proposition}
Let $X$ be a translation surface made of $2d$ triangles and $I \subset X$ an admissible
horizontal interval.
The first return map $T$ of the flow on $I$ is an interval exchange
transformation. Furthermore, if $X$ has no vertical saddle connection then $T$
satisfies the Keane condition and the number of subintervals in $T$ is $d+1$.
\end{proposition}
In the case of the torus ($d=1$) the first return maps on admissible intervals are rotations
or 2-i.e.t.
We will denote by $\cC(d)$ the set of translation surfaces obtained by
gluing $2d$ triangles \footnote{In Teichm\"uller theory, $\cC(d)$ is just the
finite union of strata of translation surfaces with given dimension. For example one
has $\cC(2) = \cH(0)$, $\cC(3) = \cH(0,0)$, $\cC(4) = \cH(0,0,0) \cup \cH(2)$
and $\cC(5) = \cH(0,0,0,0) \cup \cH(2,0) \cup \cH(1,1)$.}.

\subsection{Lagrange constants, best approximations and Delaunay triangulations}
Now we explain how Lagrange constants of interval exchange transformation can be computed
from the $\Norm$-norm of holonomies of saddle connections.

Let $X$ be a translation surface. To a saddle connection can be associated a
vector in $\C$ that is called its
\emph{holonomy}. It corresponds to the displacement induced on the translation
structure while traveling along this segment. Given a translation surface $X$
we denote by $V(X) \subset \C$ the set of holonomies of saddle connection in
$X$.
\begin{theorem}[\cite{Vorobets}, \cite{HubertMarcheseUlcigrai}]
\label{thm:a}
Let $X$ be a translation surface. We define
\[
\a(X) := \liminf_{\stackrel{v\in V(X)}{\Im(v) \to \infty}} \frac{\Norm^{2}(v)}{\area(X)},
\]
where $\Norm(v) := \sqrt{|\Re(v)|\ |\Im(v)|}$, for vectors $v \in \C$.

Let $T$ be an interval exchange transformation that is the induced map of the
translation flow of $X$ on some admissible horizontal interval.
Then $\cE(T) = \a(X)$.
\end{theorem}

In the above theorem, the minimum is taken over all saddle connections in
$V(X)$. We will show in Lemma~\ref{lem:best_approx_vs_delaunay} that it is enough to restrict the $\liminf$ to
a subset of edges. Then we show that this subset of edges is exactly
the set of edges of the Delaunay triangulations of some deformations
of $X$.

We first need to introduce "quadrants" of saddle connections
from~\cite{Marchese-Khinchin} and~\cite{DelecroixUlcigrai}. Let $X$ be a
translation surface. Given two segments with the same starting point they
have a well defined angle. If the starting point has a total angle
$2k\pi$ then the angle between the segments is between $0$ and $2k\pi$.
Let us consider the vertical half lines that start from the singularities
in direction $\sqrt{-1}$. Up to a change of orientation, a saddle connection can
always be made to start with an angle in $[-\pi/4,\pi/4]$ with respect
to one of these outgoing vertical separatrices. Let us fix a numbering
of these half lines from $1$ to $d$. We associate to each of them
a pair $V^i_\ell(X)$ and $V^i_r(X)$ of subset of saddle connections
that are the one with angle respectively in $[-\pi/4,0]$ and $[0,\pi/4]$
with the corresponding vertical.

A saddle connection is called a \emph{best approximation} if it is the diagonal
of an immersed rectangle in the surface whose boundary edges are horizontal and
vertical. Equivalently, $\gamma \in V^{i}_s(X)$ where $i \in \{1,2,\ldots,d\}$
and $s \in \{\ell,r\}$ is not a best approximation if there exists $\eta \in
V^{i}_s(X)$ (the same quadrant) so that $|\Re(\eta)| < |\Re(\gamma)|$ and
$|\Im(\eta)| < |\Im(\gamma)|$.

We will denote by $V_{ba}(X)$ the set of holonomies of best approximation on $X$. From its definition
it follows that
\[
\a(X) = \inf_{v \in V_{ba}(X)} \Norm(v) = \inf_{v \in V(X)} \Norm(v)
\]
where $\a(X)$ is defined in Theorem~\ref{thm:a}.

We now show that best approximations can be seen as the edges of some Delaunay
triangulations.
The group
$\SL(2,\R)$ acts on the set of (equivalence classes) of translation surfaces
through its linear action on $\R^2$. The subaction of the diagonal subgroup
$\displaystyle g_t = \begin{pmatrix}e^t & 0 \\ 0 & e^{-t}\end{pmatrix}$ is
called the \emph{Teichm\"uller flow}. Note that this action is also well
defined on points and segments (that is, given a pair $(X,p)$ (or $(X,\gamma)$)
made of a translation surface and a point, or a segment, and a matrix
$m$, the image of $p$ (or $\gamma$) in $m \cdot X$ is well defined).
As in the case of the plane, one can define the Delaunay triangulation of a
translation surface $X$ by considering maximal immersed squares. The following
two lemmas are elementary.
\begin{lemma}
Let $X$ be a translation surface with neither a horizontal nor a vertical saddle
connection. Then for all $t \in \R$, $g_t X$ has a well defined Delaunay
triangulation except for a discrete set of times $t_n$ for which some quadruple
of singularities are on the boundary of an immersed square (in which case
there is no uniqueness of Delaunay triangulation).
\end{lemma}

\begin{lemma} \label{lem:best_approx_vs_delaunay}
Let $X$ be a translation surface. For a saddle connection $\gamma$ in $X$ the following are equivalent:
\begin{enumerate}
\item $\gamma$ is a best approximation,
\item there exists $t \in \R$ such that the $L^\infty$-Delaunay triangulation of $g_t X$ contains $g_t \gamma$ as an edge.
\end{enumerate}
\end{lemma}

\begin{proof}
We just need to remark that any rectangle immersed in $X$ can be turned into a square using the $g_t$ action.
\end{proof}

\subsection{At the bottom of the Lagrange spectrum are golden surfaces}

Let us introduce the surfaces that will be shown to be exactly the ones that
minimize the quantity $\a(X)$ of Theorem~\ref{thm:a}. Let $\Lambda$ be the lattice
$\Lambda = \Z (-1,1) \oplus \Z (\phi-1, \phi)$ where $\phi = (\sqrt{5}+1)/2$ is
the golden ratio. We call any lattice in the family $g_t \Lambda$ a
\emph{golden lattice} and the associated quotients $\C / g_t \Lambda$ a
\emph{golden torus}. We also call any parallelogram generated by a basis of
$\Lambda$ a \emph{golden parallelogram}. A \emph{golden surface} in $\cC(d)$ is
a translation surface obtained by gluing together $d$ identical golden
parallelograms (each obtained by gluing $2$ triangles). In geometric terms, such
surface is a ramified covering of degree $d$ of a golden torus.

In this section we prove the following result
\begin{theorem} \label{thm:lagrange_surface}
There exists $\epsilon_0$ such that the following holds.
If $X$ is a translation surface in $\cC(d)$ such that
\[
\inf_{v \in V(X)} \frac{\Norm^{2}(v)}{\area(X)} \geq \frac{1}{d\sqrt{5}} - \frac{\epsilon_0}{d^2}
\]
then $X$ is a golden surface.
\end{theorem}
Together with Theorem~\ref{thm:a}, Theorem~\ref{thm:lagrange_surface} implies~Theorem~\ref{thm:lagrange}.
In order to simplify notation in the proof, we will from now on always deal with surfaces such that $\area(X) = d$. That
way the mean area of a triangle is $1/2$ independently of $d$.
\begin{lemma} \label{lem:low_to_up_bound}
There exist constants $\epsilon_1 > 0 $ and $C_1 > 1$ such that the following holds.
For any area $d$ translation surface $X$ in $\cC(d)$ such that for some $\epsilon < \epsilon_1 / d$
we have
\[
\inf_{v \in V_{ba}(X)} \Norm^{2}(v) \geq \frac{1}{\sqrt{5}} - \epsilon.
\]
Then
\[
\sup_{v \in V_{ba}(X)} \Norm^{2}(v) \leq \frac{1}{\sqrt{5}} + C_1 d \epsilon.
\]
\end{lemma}

\begin{proof}
Let us consider a translation surface $(X,\omega)$ of area $d$ and let $\epsilon < \epsilon_1/d = \frac{2\sqrt{2} - \sqrt{5}}{20d} \simeq \frac{0.03}{d}$.
Assume that 
\[
\inf_{v \in V_{ba}(X)} \Norm^{2}(v) \geq \frac{1}{\sqrt{5}} - \epsilon.
\]
As a consequence of Corollary~\ref{cor:area_triangle}, we have for any translation surface
\begin{equation} \label{eq:bound_triangle}
\inf \{\area(\Delta): \text{$\Delta$ admissible triangle in $X$}\}
\geq
\frac{\sqrt{5}}{2} \inf_{v \in V_{ba}(X)} \Norm^{2}(v)
\geq
\frac{\sqrt{5}}{2} \left( \frac{1}{\sqrt{5}} - \epsilon \right).
\end{equation}
Let us consider the $L^\infty$-Delaunay triangulation $T$ of the surface $X$
and let $\Delta_0$ be one triangle in $T$. Using the fact that
the sum of the areas of the $2d$ triangles from $T$ is $\area(X) = d$ we have that
\[
\area(\Delta_0) \leq d - (2d-1)\ \min_{\Delta \in T} \area(\Delta).
\]
Now, using~\eqref{eq:bound_triangle} we obtain that
\begin{align*}
\area(\Delta_0)
& \leq d - \left(2d-1\right) \frac{\sqrt{5}}{2} \left(\frac{1}{\sqrt{5}} -\epsilon\right) \\
& \leq \left(\frac{1}{\sqrt{5}} - \epsilon \right)
     \left( \frac{\sqrt{5}d}{1 - \sqrt{5} \epsilon} - d\sqrt{5} + \frac{\sqrt{5}}{2}\right) \\
& \leq \left( \frac{1}{\sqrt{5}} - \epsilon \right)
      \left( \sqrt{5} d (1 + 2\sqrt{5} \epsilon) - d \sqrt{5} + \frac{\sqrt{5}}{2} \right) \\
& \leq \left( \frac{1}{\sqrt{5}} - \epsilon \right)
       \left( \frac{\sqrt{5}}{2} + 10 d \epsilon \right)
  \leq \left( \frac{\sqrt{5}}{2} + 10 d \epsilon \right) \inf_{v \in V_{ba}(X)} \Norm^{2}(v).
\end{align*}
Here we used the inequality $1/(1-\sqrt{5}\epsilon) \leq 1 + 2 \sqrt{5} \epsilon$ which is valid since $\sqrt{5}\epsilon < 1/2$. From our assumption, $10d\epsilon < \sqrt{2} - \sqrt{5}/2$,  and we can apply Lemma~\ref{cor:distorsion_from_area} and get that
\[
\sup_{v \in V_{ba}(X)} \Norm^{2}(v) \leq \left(1 + \frac{\sqrt{5}+1}{2} 10 d \epsilon \right) \left( \frac{1}{\sqrt{5}} - \epsilon \right)
     \leq \frac{1}{\sqrt{5}} + C_1 d \epsilon
\]
where $C_1 = 10 ( 2\sqrt{2} + \sqrt{5}) + 1/2 \leq 52$.
\end{proof}

Let us define the following distance between vectors
\[
\dist((x_1,y_1),(x_2,y_2)) =
\left\{\begin{array}{ll}
1 &  \text{if $\sgn(x_1) \not= \sgn(x_2)$ or $\sgn(y_1) \not= \sgn(y_2)$}\\
\max\left( \left|\log\left(\frac{x_1}{x_2}\right)\right|, \left|\log\left(\frac{y_1}{y_2}\right)\right|\right) & \text{otherwise}
\end{array} \right.
\]
Note that $\dist$ is invariant under the action of diagonal invertible matrices: in other words, for any nonzero real numbers $\alpha$ and $\beta$ we have
\[
\dist((\alpha x_1, \beta y_1), (\alpha x_2, \beta y_2)) = \dist((x_1,y_1),(x_2,y_2)).
\]
Let also $\tau$ be the linear map defined by
\[
\tau(x,y) = (-x/\phi, \phi y), \qquad \text{where \ $\phi = \frac{1+\sqrt{5}}{2}$}.
\]
Note that the only lattices satisfying $\tau(\Lambda) = \Lambda$ are the golden lattices.

\begin{lemma} \label{lem:parallelogram}
There exist constants $\epsilon_2 > 0$ and $C_2 > 1$ such that for any $\epsilon < \epsilon_2$
and any $\zeta_\ell = (x_\ell, y_\ell), \zeta_r = (x_r, y_r), \zeta = (x,y) \in \C$, if the conditions
\begin{itemize}
\item $x_\ell < 0$ and $x_r > 0$,
\item $0 < y_\ell < y_r < y$,
\item all of $\Norm^2(\zeta_\ell)$, $\Norm^2(\zeta_r)$, $\Norm^2(\zeta)$, $\Norm^2(\zeta_\ell - \zeta_r)$, $\Norm^2(\zeta - \zeta_\ell)$ and $\Norm^2(\zeta - \zeta_r)$
are in between $1/\sqrt{5} - \epsilon$ and $1/\sqrt{5} + \epsilon$,
\end{itemize}
hold, then all of $\dist(\zeta_r,\, \zeta-\zeta_\ell)$, $\dist(\zeta_\ell,\, \zeta-\zeta_r)$, $\dist(\tau(\zeta_\ell), \zeta_r)$, $\dist(\tau(\zeta_r),\, \zeta-\zeta_r)$, $\dist(\zeta_\ell, \tau(\zeta-\zeta_\ell))$ are smaller than $C_2 \epsilon$.
\end{lemma}

\begin{proof}
Up to rescaling we can assume that $x_\ell = -1$. Once $x_\ell$ is fixed, the maps
\[
(y_\ell, x_r, y_r) \mapsto (\Norm^2(\zeta_\ell), \Norm^2(\zeta_r), \Norm^2(\zeta_r - \zeta_\ell))
\quad \text{and} \quad
(y_\ell, x, y) \mapsto (\Norm^2(\zeta_\ell), \Norm^2(\zeta), \Norm^2(\zeta-\zeta_\ell))
\]
have invertible derivative at the point
\[
y_\ell = \frac{1}{\sqrt{5}},\ x_r = \phi - 1,\ y_r = \frac{\phi}{\sqrt{5}},\ x = \phi - 2,\ y = \frac{\phi^2}{\sqrt{5}}.
\]
Moreover, the above point is the unique solution of $\Norm^2(\zeta_\ell) = \Norm^2(\zeta_r) = \Norm^2(\zeta) = \Norm^2(\zeta-\zeta_\ell) = \Norm^2(\zeta-\zeta_r) = \frac{1}{\sqrt{5}}$. One concludes using the inverse function theorem.
\end{proof}

Recall that $V_{ba}(X)$ denote the holonomies of best approximations. Given a
quadrant $i$ in $X$ we denote by $V_{ba}^{(i)}(X)$ the set of holonomies of
saddle connections restricted to the $i$-th quadrant.
\begin{lemma} \label{lem:ba_far_appart}
There exist constants $\epsilon_3 > 0$ and $C_3 > 0$ such that if $X$ is any area $d$ surface in $\cC(d)$ such that
\[
\forall v \in V_{ba}(X),\quad \frac{1}{\sqrt{5}} - \epsilon_3 < \Norm^2(v) < \frac{1}{\sqrt{5}} + \epsilon_3
\]
then in any quadrant $i$ of $X$ we have
\[
\forall u,v \in V^{(i)}_{ba}(X), u \not= v, \qquad \dist(u,v) > C_3.
\]
\end{lemma}

\begin{proof}
Choose an $\epsilon_3$ so that we can apply Lemma~\ref{lem:parallelogram}. In a given fixed quadrant $i$
holonomies can be identified with saddle connections. Lemma~\ref{lem:parallelogram} implies that
best approximations are far apart.
\end{proof}

\begin{proof}[Proof of Theorem~\ref{thm:lagrange_surface}]
Let $d \geq 1$ and $\epsilon > 0$ be such that
\[
\epsilon < \min\left(\frac{\epsilon_1}{d}, \frac{\epsilon_2}{C_1 d},\ \frac{\epsilon_3}{C_1 d},\ \frac{C_3}{C_1 C_2 d^2}\right).
\]
Let $X \in \cC(d)$ be a translation surface so that
\[
\inf_{v \in V_{ba}(X)} \Norm^2(v) \geq \frac{1}{\sqrt{5}} - \epsilon.
\]
We will show that $X$ is actually a golden surface. Because $\epsilon <
\epsilon_1 / d$ we can apply Lemma~\ref{lem:low_to_up_bound}. Denoting
$\epsilon' = C_1 d \epsilon$ we have that
\begin{equation} \label{eq:all_sc_close}
\forall v \in V_{ba}(X), \quad \frac{1}{\sqrt{5}} - \epsilon' \leq \Norm^2(v) \leq \frac{1}{\sqrt{5}} + \epsilon'.
\end{equation}
Morever $\epsilon'$ satisfies
\[
\epsilon' < \min\left(\epsilon_2,\ \epsilon_3,\ \frac{C_3}{C_2 d}\right).
\]

We now show that any best approximation is part of a triangulation close to a golden
triangulation. In a moment we will use a fixed point argument to show
that $X$ is itself a golden surface.

A quadrilateral $q$ in $X$ is called \emph{admissible} if there exists an
immersed rectangle $R$ with horizontal and vertical sides such that there is exactly
one vertex of $q$ on each side of $R$. It is easy to see that an admissible
quadrilateral can be decomposed into two admissible triangles in two ways by
adding either of the two diagonals of $q$. The two diagonals of an admissible
quadrilateral $q$ will be called left-right and top-bottom diagonals. Given an
admissible quadrilateral, we can identify each side by its position: bottom
right, bottom left, top right, top left. The slopes of the bottom right and top
left sides are positive and following~\cite{DelecroixUlcigrai} we will say that
they are \emph{right slanted}. Similarly the bottom left and top right sides
are \emph{left slanted}.

\begin{figure}[!ht]
\begin{minipage}{0.3\textwidth}
\centerline{\includegraphics{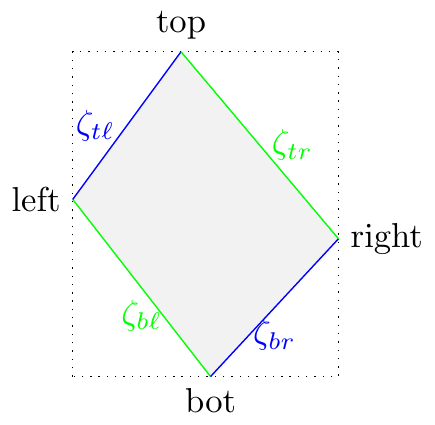}}
\subcaption{An admissible quadrilateral $q$. The sides $\zeta_{b\ell}$ and $\zeta_{tr}$ are left slanted
while  $\zeta_{br}$ and $\zeta_{t\ell}$ are right slanted.}
\label{fig:pil_pir}
\end{minipage}
\hspace{0.1\textwidth}
\begin{minipage}{0.6\textwidth}
\centerline{\includegraphics{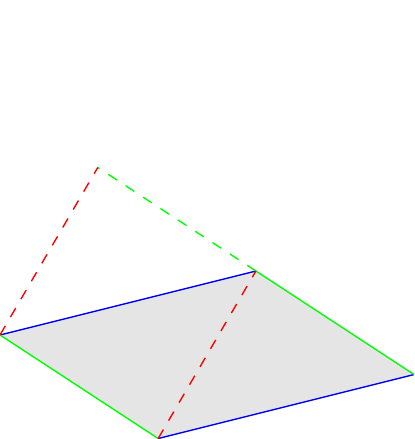} \hspace{1cm} \includegraphics{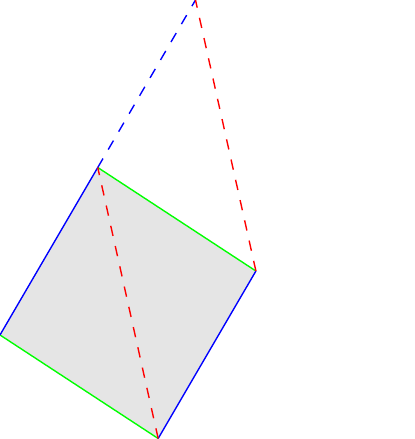}}
\subcaption{A left move followed by a right move on a golden parallelogram. The initial quadrilateral has blue right slanted and
green left slanted sides. The top-bottom diagonals are in red. The sides of the new quadrilateral are dashed.} 
\end{minipage}
\caption{Quadrangulations and diagonal changes.}
\label{fig:quad_and_moves}
\end{figure}

Let $Q$ be a quadrangulation of $X$. We denote by $E_\ell(Q)$ and $E_r(Q)$ the left
slanted and right slanted sides of $Q$. We measure how $Q$ is close to a quadrangulation
into golden parallelograms with the following function
\[
g(Q) = \max \left(
  \max_{\zeta \in E_r,\zeta' \in E_r} \dist(\zeta, \zeta'),\,
  \max_{\zeta \in E_\ell,\zeta' \in E_\ell} \dist(\zeta, \zeta'),\,
  \max_{\zeta \in E_r,\zeta' \in E_\ell} \dist(\tau(\zeta), \zeta'),\,
\right).
\]
It is clear that $g(Q) = 0$ if and only if $Q$ is a golden quadrangulation.

Let $\epsilon'' = d C_2 \epsilon'$.

\textbf{claim 1:} Let $\zeta_0$ be a right slanted best approximation in $X$. Then there exists a
unique quadrangulation $Q_0$ of $X$ into admissible quadrilaterals that admits $\zeta_0$ as one
of its sides and so that $g(Q_0) < \epsilon''$.

Since $\epsilon' < \epsilon_2$, it follows from Lemma~\ref{lem:parallelogram}
that $\zeta_0$ is the bottom right side of a unique quadrilateral so that its bottom left side $\zeta_{b\ell}$, its top
left side $\zeta_{t\ell}$ and its top right side $\zeta_{tr}$ satisfy that all of $\dist(\zeta_0, \zeta_{t\ell})$,
$\dist(\zeta_{b\ell}, \zeta_{tr})$ and $\dist(\tau(\zeta_0), \zeta_{b\ell})$ are smaller than $C_2 \epsilon'$.
We denote this quadrilateral by $q(\zeta_0)$.

Similarly, given a left slanted best approximation $\zeta$ one can also build a
unique quadrilateral so that its bottom left side $\zeta_{b\ell}$ is
$C_2\epsilon'$-close to $\tau^{-1}(\zeta)$.

Using these two rules we can build step by step a quadrangulation of $X$. More precisely, starting from $q_0$ we consider
its top sides and construct new quadrilaterals from these two. Then we repeat the operations with the newly created sides.

By construction, either the newly created top side will coincide with an
already constructed bottom side of another quadrilateral, or
we will have some non-trivial intersection between two constructed
parallelograms.  Let us show that this second case cannot happen. Let us
consider a chain $q_0, q_1, \ldots, q_k$ of adjacent quadrilaterals such
that the bottom sides of $q_0$ and $q_k$ belong to the same bundle of $X$ and
the bundles of $q_i$ for $i=0,\ldots,k-1$ are disjoint. Necessarily $k \leq d$,
and hence the bottom sides of $q_k$ are $k C_2 \epsilon'$-close to those of
$q_0$. As $k C_2 \epsilon' \leq \epsilon'' \leq C_3$,
Lemma~\ref{lem:ba_far_appart} implies that $q_0 = q_k$. This finishes the proof of the claim.

We will now use claim~1 to show that the surface $X$ is itself a golden
surface. In order to do so, we analyze how the different quadrangulations of
claim~1 are related to each other. The relation between the various
quadrangulations correspond to a particular case of the so called
Ferenczi-Zamboni induction~\cite{Ferenczi-induction} (see also~\cite{DelecroixUlcigrai}
for a particular case related to quadrangulations).

Given two admissible quadrangulations $Q$ and $Q'$ of $X$ we say that $Q'$ is obtained from $Q$ by a
\emph{left move} if the left slanted sides of $Q$ and $Q'$ are equal and the right slanted
sides of $Q'$ are the top-bottom diagonals of quadrilaterals of $Q$. We define right moves similarly. See also Figure~\ref{fig:quad_and_moves}.

By the uniqueness in claim~1, there exists a biinfinite sequence of
quadrangulations of $X$ into admissible quadrilaterals \ldots, $Q_{-1}, Q_0,
Q_1, Q_2, \ldots$ each of them satisfying the claim, and such that $Q_{n+1}$ is
obtained from $Q_n$ by a left move followed by a right move.

Recall that the bundles in $X$ are numbered. Hence each quadrilateral also
inherits a number given by the bundle it belongs to. We say that two
quadrangulations $Q$ and $Q'$ are \emph{combinatorially equivalent} if
for each $i \in \{1,\ldots,d\}$, the labels of the quadrilaterals
on the top left and top right of the quadrilateral labeled $i$ are the same
for $Q$ and $Q'$.
It is easy to see that there are finitely many possible combinatorial types
of quadrangulations. Moreover, if $Q'$ is obtained by a left or right move
from $Q$ then the combinatorial type of $Q'$ is only determined by the
combinatorial type of $Q$. We refer the reader to~\cite{DelecroixUlcigrai}
or~\cite{Ferenczi-induction} for these two elementary facts. Combining these
two facts, we see that the sequence of combinatorial types of the quadrangulations $(Q_n)_{n
\in \Z}$ is periodic for some period $p$. Since a left or right move operates as a linear
transformation on the holonomies of the sides, there exists a $2d \times 2d$ matrix $A$
with non-negative integer coefficients so that the holonomies of the sides of $Q_p$ are the images
of the side of $Q_0$ by $A$. Let $v_0$ be the holonomy of $\zeta_0$. Let $Q'_0$ be
the quadrangulation with the same combinatorics of $Q_0$ but such that all right
slanted sides have holonomy $v_0$ and all left slanted sides have holonomy $\tau(v_0)$.
The quadrangulation $Q'_0$ is a golden quadrangulation of another translation surface.
By construction, the real and imaginary parts of holonomies of $Q'_0$ are eigenvectors
of $A$ (namely real parts are multiplied by $\phi^{-2p}$ and imaginary parts by $\phi^{2p}$).
Applying the Perron-Frobenius theorem to the real and imaginary parts of the holonomies
we obtain the uniqueness of the fixed point (up to scalar multiples). Hence, $Q'_0 = Q_0$ and $X$
is a golden surface.
\end{proof}



\bibliographystyle{amsplain}


\begin{dajauthors}
\begin{authorinfo}[mb]
  Micahel Boshernitzan\\
  Rice University\\
  Houston, Texas\\
  \url{http://math.rice.edu/~michael/}
\end{authorinfo}
\begin{authorinfo}[vd]
  Vincent Delecroix\\
  CNRS, Universit\'e de Bordeaux\\
  Bordeaux, France\\
  \url{http://www.labri.fr/perso/vdelecro/}
\end{authorinfo}
\end{dajauthors}

\end{document}